\documentclass{amsart}
\usepackage{amsfonts}
\newtheorem{thm}{Theorem}[section]

\newtheorem{lem}[thm]{Lemma}
\newtheorem{prop}[thm]{Proposition}
\newtheorem{defn}[thm]{Definition}
\newtheorem{rem}[thm]{Remark}
\newtheorem{ex}[thm]{Example}

\begin{document}
\title{A characterization of nilpotent Leibniz algebras}
\author{Alice Fialowski, A.Kh. Khudoyberdiyev and B.A. Omirov}
\address{[A. Fialowski] Institute of Mathematics, E\"otv\"os Lor\'and
University, Budapest, P\'azm\'any P\'eter s\'et\'any 1/C, 1117
Hungary}
\email{fialowsk@cs.elte.hu}
\address{[A.Kh. Khudoyberdiyev -- B.A. Omirov] Institute of Mathematics, 29,
Do'rmon yo'li srt., 100125, Tashkent (Uzbekistan)}
\email{khabror@mail.ru
--- omirovb@mail.ru}
\thanks{The research of the first author was partially supported by the
grant OTKA K77757.}

\maketitle

\begin{abstract}
W. A. Moens proved that a Lie algebra is nilpotent if and only if
it admits an invertible Leibniz-derivation. In this paper we show
that with the definition of Leibniz-derivation from \cite{Moens}
the similar result for non Lie Leibniz algebras is not true.
Namely, we give an example of non nilpotent Leibniz algebra which
admits an invertible Leibniz-derivation. In order to extend the
results of paper \cite{Moens} for Leibniz algebras we introduce a
definition of Leibniz-derivation of Leibniz algebras which agrees
with Leibniz-derivation of Lie algebras case. Further we prove
that a Leibniz algebra is nilpotent if and only if it admits an
invertible Leibniz-derivation. Moreover, the result that solvable
radical of a Lie algebra is invariant with respect to a
Leibniz-derivation was extended to the case of Leibniz algebras.
\end{abstract} \maketitle




\section{Introduction}

In 1955, Jacobson \cite{Jac} proved that a Lie algebra over a
field of characteristic zero admitting a non-singular (invertible)
derivation is nilpotent. The problem, whether the inverse of this
statement is correct, remained open until work \cite{Dix}, where
an example of an nilpotent Lie algebra, whose derivations are
nilpotent (and hence, singular), was constructed. Such types of
Lie algebras are called characteristically nilpotent Lie algebras.

The study of derivations of Lie algebras lead to appearance of
natural generalization -- pre-derivations of Lie algebras
\cite{Mul}. In \cite{Baj} it is proved that Jacobson's result is
also true in terms of pre-derivations. Similar to the example of
Dixmier and Lister \cite{Dix} several examples of nilpotent Lie
algebras, whose pre-derivations are nilpotent were presented in
\cite{Baj}, \cite{Burd}. Such Lie algebras are called strongly
nilpotent \cite{Burd}.

In paper \cite{Moens} a generalized notion of derivations and
pre-derivation of Lie algebras is defined as Leibniz-derivation of
order $k$. In fact, a Leibniz-derivation is a derivation of a
Leibniz $k$-algebra constructed by Lie algebra \cite{Cas1}.

Below we present the characterization of nilpotency for Lie
algebras in terms of Leibniz-derivations.

\begin{thm}\cite{Moens} \label{t5} A Lie algebra over a field of characteristic zero is nilpotent
if and only if it has an invertible Leibniz-derivation.
\end{thm}

Leibniz algebras were introduced by Loday in \cite{Lo 1}-\cite{Lo
2} as a non-antisymmetric version of Lie algebras. Many results of
Lie algebras are extended to Leibniz algebras case. Since the
study of derivations and automorphisms of a Lie algebra plays
essential role in the structure theory, the natural question
arises whether the corresponding results for Lie algebras can be
extended to more general objects.

In \cite{Lad} it is proved that a finite dimensional complex
Leibniz algebra admitting a non-singular derivation is nilpotent.
Moreover, it was shown that similarly to the case of Lie algebras,
the inverse of this statement does not hold and the notion of
characteristically nilpotent Lie algebra can be extended for
Leibniz algebras \cite{Omir}.
\medskip

In this paper we show that if we define Leibniz-derivations for
Leibniz algebra as in \cite{Moens}, then Theorem \ref{t5} does not hold.
In order to avoid the confusion we need to modify the notion
of Leibniz-derivation for Leibniz algebras.

Recall, in the definition of Leibniz-derivation of order $k$ for
Lie algebras the $k$-ary bracket is defined as multiplication of
$k$ elements on the left side. For the case of Leibniz algebras we
propose the definition of Leibniz-derivation of order $k$ as
$k$-ary bracket on the right side. Due to anti-commutativity of
multiplication in Lie algebras this definition agrees with the
case of Lie algebras.

Note that a vector space equipped with right sided $k$-ary
multiplication is not a Leibniz $k$-algebra defined in
\cite{Cas1}. For Leibniz-derivation of Leibniz algebra we prove
the analogue of Theorem \ref{t5} for finite dimensional Leibniz
algebras over a field of characteristic zero.

Through the paper all spaces an algebras are assumed finite
dimensional.

\section{Preliminaries}
In this section we present some known facts
about Leibniz algebras and Leibniz $n$-algebras.

\begin{defn}
A vector space $L$ over a field $F$ with a binary operation
$[-,-]$ is a (right) Leibniz algebra, if for any $x, y, z \in L$
the so-called Leibniz identity
  $$
[x, [y, z]] = [[x, y], z] - [[x, z], y]
$$
 holds.
\end{defn}

Every Lie algebra is a Leibniz algebra, but the bracket in a
Leibniz algebra needs not to be skew-symmetric.

For a Leibniz algebra $L$ consider the following central lower and
derived sequences:
$$
L^1=L,\quad L^{k+1}=[L^k,L^1], \quad k \geq 1,
$$
$$L^{[1]} = 1, \quad L^{[s+1]} = [L^{[s]}, L^{[s]}], \quad s \geq 1.$$

\begin{defn} A Leibniz algebra $L$ is called
nilpotent (respectively, solvable), if there exists  $p\in\mathbb N$ $(q\in
\mathbb N)$ such that $L^p=0$ (respectively, $L^{[q]}=0$).
\end{defn}

Levi's theorem, which has been proved for left Leibniz
algebras in \cite{Bar}, is also true for right Leibniz algebras.

\begin{thm} (Levi's Theorem).\label{tbar}
Let $L$ be a Leibniz algebra over a field of characteristic zero
and $R$ be its solvable radical. Then there exists a semisimple
subalgebra Lie $S$ of $L$, such that $L=S\dot{+}R.$
\end{thm}
\medskip

The following theorem from linear algebra characterizes the
decomposition of a vector space into the direct sum of characteristic
subspaces.

\begin{thm}\cite{Cur} Let $A$ be a linear transformation
of the vector space $V.$ Then $V$ decomposes into the direct sum of
characteristic subspaces $V=V_{\lambda_1}\oplus V_{\lambda_2}
\oplus \dots \oplus V_{\lambda_k}$  with respect to $A,$ where
$V_{\lambda_i}=\{x\in V\ | \ (A-\lambda_i I)^k(x)=0 \textrm{ for
some } k\in \mathbb{N}\}$ and $\lambda_i, 1\leq i \leq k$, are
eigenvalues of $A.$
\end{thm}

In Leibniz algebras a derivation is defined as follows

\begin{defn}
A linear transformation $d$ of a Leibniz algebra $L$ is a
derivation if for any $x, y\in L$
$$d([x,y])=[d(x),y]+[x, d(y)].$$
\end{defn}
\smallskip

Consider for an arbitrary element $x\in L$ the operator of
right multiplication $R_x:L\to L$, defined by $R_x(z)=[z,x].$
Operators of right multiplication are derivations of the Leibniz algebra
$L.$ The set $R(L)=\{R_x\ |\ x\in L\}$ is a Lie algebra with
respect to the commutator, and the following identity holds:
$$R_xR_y-R_yR_x=R_{[y,x]}.\eqno (2.1)$$

A subset $S$ of an associative algebra $A$ over a field $F$ is
called a {\it weakly closed subset} if for every pair $(a,b)\in
S\times S$ there is an element $\gamma_{(a,b)}\in F$ such that
$ab+\gamma_{(a,b)}ba\in S.$

 We will need the following result concerning weekly closed sets

\begin{thm}\cite{Jac} \label{closed} Let $S$ be a weakly closed subset
of the associative algebra $A$ of linear transformations of a
vector space $V$ over $F.$ Assume that every $W\in S$ is
nilpotent, that is, $W^k=0$ for some positive integer $k.$ Then
the enveloping associative algebra $S^*$ of $S$ is nilpotent.
\end{thm}

The classical Engel's theorem for Lie algebras has the following analogue
for Leibniz algebras.

\begin{thm}\cite{Ayu}\label{t1}
A Leibniz algebra $L$ is nilpotent if and only if $R_x$ is nilpotent for any
$x \in L.$
\end{thm}

The following Theorem generalizes Jacobson's theorem to
Leibniz algebras.

\begin{thm}\cite{Lad}\label{t2}
Let $L$ be a complex Leibniz algebra which admits a non-singular
derivation. Then $L$ is nilpotent.
\end{thm}

The next example presents $n$-dimensional Leibniz algebra
possessing only nilpotent derivations.
\begin{ex}\label{excarnil} Let $L$ be an $n$-dimensional Leibniz algebra and let
$\{e_1, e_2, \dots , e_n\}$ be a basis of $L$ with the following
table of multiplication:
$$\left\{\begin{array}{ll}
[e_1,e_1]=e_{3},&  \\[1mm]
[e_i,e_1]=e_{i+1}, &  2\leq i \leq n-1,\\[1mm]
[e_1,e_2]= e_4, & \\[1mm]
[e_i,e_2]= e_{i+2},& 2\leq i \leq n-2,\\[1mm]
\end{array} \right.$$
(omitted products are equal to zero).

Using derivation property it is easy to see that every derivation
of  $L$ has the following matrix form:
$$\left(\begin{array}{cccccccc}
0& 0& a_3&a_4&a_5&\dots&a_{n-1}&a_n\\[1mm]
0& 0& a_3&a_4& a_5&\dots& a_{n-1}&b_n\\[1mm]
0& 0& 0& a_3& a_4& \dots & a_{n-2}&a_{n-1}\\[1mm]
\dots & \dots & \dots& \dots &\dots&\dots &\dots &\dots \\[1mm]
0& 0& 0& 0& 0 &\dots &0 &a_3\\[1mm]
0& 0& 0& 0& 0& \dots &0  &0 \\[1mm]
\end{array} \right).$$

Thus, every derivation of $L$ is nilpotent, i.e., $L$ is characteristically nilpotent.
\end{ex}

Let us give the definition of Leibniz $n$-algebras.
\begin{defn}\cite{Cas1}
 A vector space $\mathcal{L}$ with an $n$-ary multiplication
$[-,-,...,-]: \mathcal{L}^{\otimes n} \to \mathcal{L}$ is a
Leibniz $n$-algebra if it satisfies the following identity:
$$ [[x_1,x_2, \dots ,x_n],y_2,\dots ,y_n]=\sum_{i=1}^n [x_1,\dots ,
x_{i-1},[x_i,y_2,\dots,y_n],x_{i+1},\dots,x_n]. \eqno(2.2)
$$
\end{defn}

Let $L$ be a Leibniz algebra with the product $[-,-].$ Then the
vector space $L$ can be equipped with a Leibniz $n$-algebra
structure with the following product: $$[x_1, x_2, \dots, x_n] =
[x_1, [x_2, \dots, [x_{n-1},x_n]]].$$

\begin{defn}
A derivation of a Leibniz $n$-algebra $\mathcal{L}$ is a
$\mathbb{K}$-linear map $d: \mathcal{L}\rightarrow \mathcal{L}$
satisfying
$$d([x_1, x_2, \dots, x_n ]) = \sum\limits_{i=1}^n[x_1, \dots, d(x_i),
\dots, x_n ].$$
\end{defn}

The notion of Leibniz-derivation of Lie algebra was introduced in
\cite{Moens} and it generalizes the notions of derivation and
pre-derivation of Lie algebra.

\begin{defn}\label{d1} A Leibniz-derivation of order $n$ for a Lie algebra $G$ is an
endomorphism $P$ of $G$ satisfying the identity
$$P([x_1, [x_2, \dots, [x_{n-1}, x_n ]]]) = [P(x_1), [x_2, \dots, [x_{n-1}, x_n
]]] +$$$$+ [x_1, [P(x_2), \dots, [x_{n-1}, x_n ]]] + \dots + [x_1,
[x_2, \dots, [x_{n-1}, P(x_n)]]]$$ for every $x_1, x_2, \dots, x_n
\in G.$
\end{defn}
In other words, a Leibniz-derivation of order $n$ for a Lie
algebra $G$ is a derivation of $G$ viewed as a Leibniz $n$-algebra.

\bigskip

\section{Leibniz-derivation of Leibniz algebras}

The following example shows that Definition \ref{d1} is not
substantial for the case of Leibniz algebras.

\begin{ex}\label{e1} Let $R$ be an $(n+1)$-dimensional solvable Leibniz algebra
and $\{e_1, e_2, \dots, e_n, e_{n+1}\}$ be a basis of $R$ with the table of multiplication given by
$$\left\{\begin{array}{lll}
\, [e_1,e_1]=e_3, & [e_i,e_1]=e_{i+1}, & 2\leq i\leq n-1,\\
\, [e_1,e_{n+1}]=e_2+\sum\limits_{i=4}^{n-1}\alpha_ie_i,   & [e_2,e_{n+1}]=e_2+\sum\limits_{i=4}^{n-1}\alpha_ie_i,& \\
\, [e_i,e_{n+1}]=e_i+\sum\limits_{j=i+2}^n\alpha_{j-i+2}e_i, & 3\leq i\leq n, &\\
\end{array}\right.$$
\end{ex}
It is easy to see that $[R,[R,R]]=0.$ For the identity map $d$ we
have $$0= d([x,[y,z]]) = [d(x),[y,z]] + [x,[d(y),z]] +
[x,[y,d(z)]] = 0.$$

Therefore, the invertible map $d$ satisfies the condition of
Definition \ref{d1}, but the Leibniz algebra $R$ is not nilpotent,
i.e. analogue of Theorem \ref{t5} for Leibniz algebras is not
true.

\begin{rem} The Example \ref{e1} can be extended for any non nilpotent
solvable Leibniz algebra $L$ such that $L^2$ lies in the right
annihilator of $L$.
\end{rem}

Let us introduce $n$-ary multiplication as follows
$$[x_1, x_2, \dots, x_n]_r = [[[x_1, x_2], x_3] \dots, x_n].$$

The next example shows that, in general, a vector space equipped with defined
$n$-ary multiplication $[x_1, x_2, \dots, x_n]_r$ is not a Leibniz
$n$-algebra.

\begin{ex}\label{e3} Let $R$ be a solvable Leibniz algebra and let
$\{e_1, e_2, \dots, e_n, x\}$ be a basis of $R$ such that
multiplication table of $R$ in this basis has the following form
\cite{Cas2}:
$$\left\{\begin{array}{ll}
\, [e_i,e_1]=e_{i+1}, & 1\leq i\leq n-1,\\
\, [x,e_1]=e_1, & \\
\, [e_i,x]=-ie_i, & 1\leq i\leq n.\\
\end{array}\right.$$

It is not difficult to check that the vector space $R$ with
$k$-ary multiplication $[x_1, x_2, \dots, x_k]_r$ does not define
Leibniz $k$-algebra structure. Indeed, we have
$$ [[e_1,e_1, \dots ,e_1]_r,x,\dots ,x]_r=[\dots [[
[ \underbrace{\dots[[e_1, e_1],e_1], \dots, e_1}\limits_{k -
times},] \underbrace{x], x], \dots, x}\limits_{k-1 - times}]=
[\dots [[ e_k, \underbrace{x], x], \dots, x}\limits_{k-1 - times}]
= (-k)^{k-1}e_k.$$

On the other hand
$$\sum_{i=1}^k [e_1,\dots ,
e_{1},\underbrace{[e_1,x,\dots,x]_r}\limits_{i -
th},e_{1},\dots,e_1]_r =\sum_{i=1}^k [e_1,\dots ,
e_{1},\underbrace{(-1)^{k-1}e_1}\limits_{i -
th},e_{1},\dots,e_1]_r =$$$$= (-1)^{k-1}\sum_{i=1}^k
[\underbrace{\dots[[e_1, e_1],e_1], \dots, e_1}\limits_{k -
times}] = (-1)^{k-1}\sum_{i=1}^k e_k = (-1)^{k-1} k e_k.
$$
Hence identity (2.2) does not hold for $k \geq 3.$
\end{ex}

Now we define the notion of Leibniz-derivation for Leibniz
algebras.

\begin{defn}\label{d2}
A Leibniz-derivation of order $n \in \mathbb{N}$ for a Leibniz
algebra $L$ is a $\mathbb{K}$-linear map $d: L\rightarrow L$
satisfying
$$d([x_1, x_2, \dots, x_n ]_r) = \sum\limits_{i=1}^n[x_1, \dots,
d(x_i), \dots, x_n ]_r.$$
\end{defn}

\begin{prop} For Lie algebras Definition \ref{d2} agrees with Definition
\ref{d1}.
\end{prop}

\begin{proof} Let $L$ be a Lie algebra, then we have
 $$P([x_1, [x_2, \dots, [x_{n-1}, x_n ]]]) =(-1)^n
P([[[x_n, x_{n-1}], \dots, x_2],x_1])= (-1)^nP([x_n, x_{n-1},
\dots, x_1 ]_r).$$
On the other hand,
$$[P(x_1), [x_2, \dots, [x_{n-1}, x_n ]]] + [x_1, [P(x_2),
\dots, [x_{n-1}, x_n ]]] + \dots + [x_1, [x_2, \dots, [x_{n-1},
P(x_n)]]]=$$
$$=(-1)^n[[[x_n, x_{n-1}],\dots, x_2],P(x_1)] + (-1)^n[[[x_n, x_{n-1}],\dots, P(x_2)],x_1] +
\dots + (-1)^n[[[P(x_n), x_{n-1}],\dots, x_2],x_1]=$$
$$=(-1)^n\left([x_n, x_{n-1},\dots, x_2,P(x_1)]_r + [x_n, x_{n-1},\dots, P(x_2),x_1]_r +
\dots + [P(x_n), x_{n-1},\dots, x_2,x_1]_r\right)=$$
$$=(-1)^n \sum\limits_{i=1}^n[x_n, \dots, P(x_i), \dots, x_1 ]_r.$$

This implies the equality
$$P([x_1, [x_2, \dots, [x_{n-1}, x_n ]]]) = [P(x_1), [x_2, \dots, [x_{n-1}, x_n
]]] +$$$$+ [x_1, [P(x_2), \dots, [x_{n-1}, x_n ]]] + \dots + [x_1,
[x_2, \dots, [x_{n-1}, P(x_n)]]],$$ which is equivalent to
$$P([x_n, x_{n-1},
\dots, x_1 ]_r)= \sum\limits_{i=1}^n[x_n, \dots, P(x_i), \dots,
x_1 ]_r. $$

Relabeling $x_i$ with $x_{n+1-i}$ for $1 \leq i \leq n$ we
complete the proof of the proposition.
\end{proof}

Let $LDer_n(L)$ denote the set of all Leibniz-derivations of order
$n$ for a Leibniz algebra $L$ and let $LDer(L)$ be the set of all
Leibniz-derivations, i.e. $LDer(L) = \bigcup_{n\in
\mathbb{N}}LDer_n(L).$

Note that a Leibniz derivation of order $2$ is a derivation.
Moreover, any derivation is a Leibniz-derivation of any order $n.$
Thus, the order of a Leibniz-derivation is not unique.

\begin{lem} The following statements are true

1) If $s, t \in \mathbb{N}$ and $s| t,$ then $LDer_{s+1}(L)
\subset LDer_{t+1}(L);$

2) for any $k, l \in \mathbb{N},$  $LDer_{k}(L) \cap LDer_{l}(L)
\subset LDer_{k+l-1}.$
\end{lem}
\begin{proof} The proof is similar to that of Lemma 2.3 \cite{Moens}.
\end{proof}

Similarly to the case of Lie algebras we call a Leibniz-derivation
of order $3$ a pre-derivation of Leibniz algebra. A nilpotent
Leibniz algebra is called \emph{strongly nilpotent} if all its
Leibniz pre-derivations are nilpotent.

Note that a strongly nilpotent Leibniz algebra is characteristically
nilpotent, but the inverse is not true in general.

\begin{ex} Any pre-derivation of the characteristically nilpotent Leibniz
 algebra in Example \ref{excarnil} with $n=6$ have the matrix form:
$$\left(\begin{array}{cccccc}
a_1& a_1& a_3&a_4&a_5&a_6\\[1mm]
0& 2a_1& a_3&a_4&b_5&b_6\\[1mm]
0& 0& 3a_1&-a_1 + a_3&c_5&c_6\\[1mm]
0& 0& 0& 4a_1&2a_1 + a_3&a_4\\[1mm]
0& 0& 0& 0& 5a_1 & a_1 + a_3\\[1mm]
0& 0& 0& 0& 0& 6a_1\\[1mm]
\end{array} \right)$$
Thus, this Leibniz algebra is not strongly nilpotent.
\end{ex}

\begin{prop} The Leibniz algebra $L$ in Example \ref{excarnil} is strongly nilpotent
 if $n>6$.
\end{prop}

\begin{proof}
Let $d: L \rightarrow L$ be a pre-derivation of $L.$

Put
$$d(e_1) = \sum_{i=1}^{n}a_ie_i, \
d(e_2) = \sum_{i=1}^{n}b_ie_i, \ d(e_3) =\sum_{i=1}^{n}c_ie_i.$$

Consider the property of pre-derivation
$$d(e_4) =d([e_1, e_1, e_1]_r)= (3a_1+a_2)e_4 + (a_3+2a_2)e_5 +
\sum_{i=4}^{n-2}a_ie_{i+2}.$$

On the other hand,
$$d(e_4) =d([e_2, e_1, e_1]_r)= (2a_1+b_1+b_2)e_4 + (b_3+2a_2)e_5 +
\sum_{i=4}^{n-2}b_ie_{i+2}.$$

Comparing coefficients of basis elements we have
$$b_1 +b_2=a_1 + a_2, \quad b_i=a_i, \quad 3 \leq i \leq n-2.$$
The equality $d([e_1, e_1, e_3]_r) = 0$ implies $0 = c_1e_4
+c_2e_5,$ hence $c_1 = c_2= 0.$
\smallskip

The chain of equalities
$$b_1e_4 + (2a_1 + a_2+b_2)e_5 + (a_3 + a_2)e_6 + \sum_{i=4}^{n-3}a_ie_{i+3}=d([e_1, e_2, e_1]_r)=
$$ $$d(e_5) = d([e_3, e_1, e_1]_r) = (2a_1 + c_3)e_5 + (2a_2 + c_4)e_6
+ \sum_{i=5}^{n-2}c_ie_{i+2}.$$ deduce
$$b_1 =0, \quad c_3 = a_2+b_2, \quad c_4 = a_3 - a_2, \quad c_i=a_{i-1}, \quad 4 \leq i \leq n-2.$$

From the equalities
$$(3a_1+3a_2)e_6 + \sum_{i=3}^{n-4}a_ie_{i+4}=
d([e_1, e_2, e_2]_r)=d(e_6) =$$ $$d([e_4, e_1, e_1]_r)=
(5a_1+a_2)e_6 + (4a_2+a_3)e_7 + \sum_{i=4}^{n-4}a_ie_{i+4}$$
 we get
$a_1 =a_2 =0.$

Since $b_2 = a_1+a_2$ and $c_3=a_2+b_2$, we have $b_2 = c_3 =0.$

Thus we obtain
$$d(e_1) = \sum_{i=3}^{n}a_ie_i, \ d(e_2) = \sum_{i=3}^{n-2}a_ie_i+b_{n-1}e_{n-1}+ b_ne_n, \ d(e_3) = \sum_{i=3}^{n-3}a_ie_{i+1} + c_{n-1}e_{n-1}+ c_ne_n.$$

Finally, from the expression $d([e_{i-2}, e_1, e_1]_r)$ we derive
$$d(e_i) = a_3e_{i+1} + a_4e_{i+2} + \dots + a_{n+2-j}e_n,\ i \geq 4$$ which completes
the proof of Proposition.
\end{proof}

Below we present $7$- and $8$-dimensional characteristically
nilpotent Leibniz algebras, which are not strongly nilpotent.

\begin{ex} The $7$-dimensional Leibniz algebra with table of
multiplication:
$$\left\{\begin{array}{ll}
[e_1,e_1]=e_{3},&  \\[1mm]
[e_i,e_1]=e_{i+1}, &  2\leq i \leq 6,\\[1mm]
[e_1,e_2]= e_4-2e_5, & \\[1mm]
[e_i,e_2]= e_{i+2}-2e_{i+3},& 2\leq i \leq 4,\\[1mm]
[e_5,e_2]= e_7\\[1mm]
\end{array} \right.$$
is characteristically nilpotent, but not strongly nilpotent.
\end{ex}

\begin{ex} The $8$-dimensional filiform Leibniz algebra with table
of multiplication:
$$\left\{\begin{array}{ll}
[e_1,e_1]=e_{3},&  \\[1mm]
[e_i,e_1]=e_{i+1}, &  2\leq i \leq 6,\\[1mm]
[e_1,e_2]= e_4-2e_5+5e_6, & \\[1mm]
[e_i,e_2]= e_{i+2}-2e_{i+3}+5e_{i+4},& 2\leq i \leq 4,\\[1mm]
[e_5,e_2]= e_7-2e_8\\[1mm]
[e_6,e_2]= e_8\\[1mm]
\end{array} \right.$$
is characteristically nilpotent, but not strongly nilpotent.
\end{ex}

Following the proofs of Lemmas in \cite{Fi} and  \cite{Cam} for
derivations of Lie and Leibniz $n$-algebras respectively, we get
the following statement for Leibniz-derivations of order $n$ of
Leibniz algebras.

\begin{lem} \label{derivacion}
For a Leibniz-derivation $d:L  \rightarrow L$ of order $n$ of a
Leibniz algebra $L$ over a field of characteristic zero, the
following formula holds for any $k \in \mathbb{N}$:
$$d^k([x_1,\dots,x_n]_r) = \sum_{i_1+i_2+\cdots+i_n=k}\frac{k!}{i_1!i_2!\dots
i_n!}[d^{i_1}(x_1),d^{i_2}(x_2),\dots,d^{i_n}(x_n)]_r \eqno(3.1).
$$
\end{lem}

\section{Nilpotent Leibniz algebras}

Starting with a Leibniz algebra $L$, we denote the
$n$-ary algebra with multiplication
$[-,-, \dots, -]_r$ by $\mathcal{L}_n(L)$.
A subalgebra $I$ is called an $n$-ideal of $L$ or an ideal of $\mathcal{L}_n(L),$
if it satisfies
$$\sum\limits_{i=1}^n[L, \dots, I, \dots, L]_r \subseteq I.$$

Let $M$ be any Leibniz subalgebra of $L.$
Consider the following sequences:
$$\mathcal{L}_n^1(M)= M, \quad \mathcal{L}_n^{k+1}(M) =
[\mathcal{L}_n^k(M), M, \dots, M]_r, \quad k \geq 1,$$
$$\mathcal{L}_n^{[1]}(M)= M, \quad \mathcal{L}_n^{[s+1]}(M) =
[\mathcal{L}_n^{[s]}(M), \mathcal{L}_n^{[s]}(M) \dots,
\mathcal{L}_n^{[s]}(M)]_r, \quad s \geq 1.$$

\begin{defn} A Leibniz algebra $L$ is called $n$-nilpotent ($n$-solvable) if
there exists a natural number $p \in \mathbb{N} \ (q \in
\mathbb{N})$ such that $\mathcal{L}_n^p(L)=0$
($\mathcal{L}_n^{[q]}(L)=0$).
\end{defn}

\begin{lem}\label{l11} Let $M$ be an ideal of $L.$ The following inclusions
 are true

$$\mathcal{L}_n^{[k]}(M)  \subseteq M^{[k]}, \quad  \quad \mathcal{L}_n^{k}(M)  \subseteq
M^{k}.$$
\end{lem}
\begin{proof} It is easy to check that $M^k$ and $M^{[k]}$ are also ideals of
$L$ for any $k$. We shall proof the first embedding by induction
on $k$ for any $n.$

If $k=2,$ then $$\mathcal{L}_n^{[2]}(M) = [M, M, M, \dots, M]_r =
[[[M, M], M], \dots, M] = [[M^{[2]}, M], \dots, M] \subseteq
M^{[2]}.$$

Suppose that the statement holds for some $k$ and we will prove it
for $k+1.$
$$\mathcal{L}_n^{[k+1]}(M) = [[[\mathcal{L}_n^{[k]}(M),
\mathcal{L}_n^{[k]}(M)], \mathcal{L}_n^{[k]}(M)], \dots,
\mathcal{L}_n^{[k]}(M)] \subseteq$$ $$\subseteq[[[M^{[k]},
M^{[k]}],M^{[k]}], \dots, M^{[k]}] = [[M^{[k+1]},M^{[k]}], \dots,
M^{[k]}] \subseteq M^{[k+1]}.$$

The second inclusion is established in a similar way.
 \end{proof}

\begin{lem}\label{l3} $M^{[tk+1]} \subseteq \mathcal{L}_n^{[k+1]}(M),$ where
$k\in \mathbb{N}$ and $t$ is a natural number such that $2^t \geq
n.$
\end{lem}

\begin{proof} Since $M^{[p]} \subseteq M^{[p+q]}$ for any $p, q \in
\mathbb{N},$ it is sufficient to prove embedding for the minimal
$t$ such that $2^t \geq n.$

We shall use induction. If $n=3$ then $t=2.$

For $k=1$ we have $$M^{[3]} = [M^{[2]}, M^{[2]}] = [M^{[2]},
[M^{[1]}, M^{[1]}]] \subseteq [[M^{[2]}, M^{[1]}], M^{[1]}]
\subseteq [M^{[1]}, M^{[1]}, M^{[1]}]_r \subseteq
\mathcal{L}_3^{[2]}(M).$$

Suppose that the statement holds for some $k$ and we will prove it for $k+1.$
$$M^{[2(k+1)+1]} = M^{[2k+1 + 2]} = [[M^{[2k+1]}, M^{[2k+1]}] , [M^{[2k+1]}, M^{[2k+1]}]] \subseteq
$$
$$\subseteq [[M^{[2k+1]}, M^{[2k+1]}], M^{[2k+1]}]\subseteq
[\mathcal{L}_3^{[k+1]}(M), \mathcal{L}_3^{[k+1]}(M), \mathcal{L}_3^{[k+1]}(M)]_r = \mathcal{L}_3^{[k+2]}(M).$$

Let us prove the statement for any $n.$

Since $2^t \geq n$ for $k=1$ we get
$$M^{[t+1]} \subseteq M^{2^t} = \underbrace{[[M^{[1]}, M^{[1]}], \dots , M^{[1]}]}\limits_{2^t - times}
\subseteq \underbrace{[[M^{[1]}, M^{[1]}], \dots , M^{[1]}]}\limits_{n - times}
 = \mathcal{L}_n^{[2]}(M).$$

The following chain equalities and inclusions
$$M^{[t(k+1)+1]} = M^{[tk+1 + t]} = (M^{[tk+1]})^{[t+1]} \subseteq (M^{[tk+1]})^{2^t} =
\underbrace{[[[M^{[tk+1]}, M^{[tk+1]}], \dots ,
M^{[tk+1]}]}\limits_{2^t - times}\subseteq$$
$$
\subseteq \underbrace{[[M^{[tk+1]}, M^{[tk+1]}], \dots , M^{[tk+1]}]}\limits_{n - times}
 \subseteq \underbrace{[[\mathcal{L}_n^{[k+1]}(M), \mathcal{L}_n^{[k+1]}(M)], \dots , \mathcal{L}_n^{[k+1]}(M)]}\limits_{n - times}
= \mathcal{L}_n^{[k+2]}(M)$$ complete the proof of the lemma.
\end{proof}
\medskip

Further we shall need the following lemma.

\begin{lem}\label{l4} $M^{nk-k+1} = \mathcal{L}_n^{k+1}(M).$
\end{lem}

\begin{proof} The proof goes again by induction on $k$ for any $n.$

If $k=1,$ then
$$M^{n} = \underbrace{[\dots [[M, M], M], \dots, M]}\limits_{n - times}
= [M, M, \dots, M]_r = \mathcal{L}_n^{2}(M).$$

Applying induction  in the equalities
$$M^{n(k+1)-k-1+1} = M^{nk-k+1+n-1} = [\dots[[M^{nk-k+1}, \underbrace{M],M], \dots, M]}\limits_{n-1 -
times}=
$$
$$=[M^{nk-k+1}, M, \dots, M]_r =[\mathcal{L}_n^{k+1}(M), M, \dots,
M]_r = \mathcal{L}_n^{k+2}(M)
$$
we complete the proof of the lemma. \end{proof}

We denote by

$R-$ solvable radical of $L,$ i.e. the maximal solvable ideal of
the Leibniz algebra $L;$
\smallskip

$R_n-$ $n$-solvable radical of $L,$ i.e. the maximal $n$-solvable
ideal of the $n$-ary algebra $\mathcal{L}_n(L);$
\smallskip

$N-$ nilradical of $L,$ i.e. the maximal nilpotent ideal of the
Leibniz algebra $L;$
\smallskip

$N_n-$  $n$-nilradical of $L,$ i.e. the maximal $n$-nilpotent
ideal of the $n$-ary algebra $\mathcal{L}_n(L).$

\begin{prop}\label{p1} For a Leibniz algebra $L$ we have $R = R_n.$
\end{prop}

\begin{proof}  Lemma \ref{l11} implies that any solvable ideal of $L$
is also $n$-solvable. Therefore, it is sufficient to prove the
inclusion $R_n \subseteq R.$ From Lemma \ref{l3} it follows that
$R_n$ is a solvable subalgebra of $L$. Thus, we need to
prove that $R_n$ is an ideal of $L.$ According to Theorem \ref{tbar},
we can write $L = S \oplus R,$ where $S$ is a simple Lie algebra,
$R$ is a solvable ideal. Let $\pi: L \rightarrow S$ be the natural
quotient map.

Since $\pi$ is a morphism of $L,$ we have
$$\pi([L, L, \dots, L, R_n]_r) = \pi([[[[L,L], L], \dots, L], R_n]) =$$$$=
[[[[\pi(L),\pi(L)], \pi(L)], \dots, \pi(L)], \pi(R_n)] = [[[[S,S],
S], \dots, S], \pi(R_n)] = [S, \pi(R_n)].$$ On the other hand,
$$\pi([L, L, \dots, L, R_n]_r) \subseteq \pi(R_n).$$

Hence, $[S, \pi(R_n)] \subseteq \pi(R_n).$ Taking into account
that $S$ is a Lie algebra we obtain $[\pi(R_n), S] \subseteq
\pi(R_n).$ Therefore, $\pi(R_n)$ is an ideal of $S.$ Since $R_n$
is an $n$-solvable ideal of $L,$ $\pi(R_n)$ is an $n$-solvable ideal of
$S$, consequently $\pi(R_n)$ is a solvable ideal (because $\pi(R_n)$
is an ideal of $S$).

Due to semisimplicity of $S$ we get $\pi(R_n) = 0,$ which implies
$R_n \subseteq R.$
\end{proof}

\begin{lem} \label{inclusion} Let $I$ be an ideal of the
Leibniz algebra $L$ and  $d \in LDer_n(L)$ a Leibniz-derivation for
 some $n \in \mathbb{N}.$
Then $$\mathcal{L}_n^{[k]}(d(I))
\subseteq I+d^{n^{k-1}}(\mathcal{L}_n^{[k]}(I))$$ for all $k\in \mathbb{N}.$
\end{lem}
\begin{proof} Evidently $d(I)\subseteq
I+d(I)$ holds. For $k=2$, using (3.1), we have
$$\mathcal{L}_n^{[2]}(d(I)) = [d(I), d(I), \dots, d(I)]_r
\subseteq d([I, I, \dots, I]_r) + $$
$$+\sum_{\begin{array}{c} i_1+i_2+\cdots+i_n=n\\\exists i_j=0\end{array}}\frac{n!}{i_1!i_2!\dots
i_n!}[d^{i_1}(I), \dots, d^{i_j-1}(I), I, d^{i_j+1}(I),
\dots,d^{i_n}(I)]_r\subseteq$$
$$\subseteq I + d^{n}(\mathcal{L}_n^{[2]}(I)).$$

\smallskip

Assume that $\mathcal{L}_n^{[k]}(d(I)) \subseteq
I+d^{n^{k-1}}(\mathcal{L}_n^{[k]}(I)).$ Again using (3.1),
we verify the inclusion for $k+1:$

$$\mathcal{L}_n^{[k+1]}(d(I)) = [\mathcal{L}_n^{[k]}(d(I)), \mathcal{L}_n^{[k]}(d(I)), \dots, \mathcal{L}_n^{[k]}(d(I))]_r
 \subseteq$$
 $$ \subseteq [I+d^{n^{k-1}}(\mathcal{L}_n^{[k]}(I)), I+d^{n^{k-1}}(\mathcal{L}_n^{[k]}(I)), \dots, I+d^{n^{k-1}}(\mathcal{L}_n^{[k]}(I))]_r
\subseteq$$
 $$  \subseteq I + d^{n^{k}}([\mathcal{L}_n^{[k]}(I), \mathcal{L}_n^{[k]}(I), \dots, \mathcal{L}_n^{[k]}(I)]_r)
 \subseteq I + d^{n^{k}}(\mathcal{L}_n^{[k+1]}(I))).
$$

\end{proof}

\begin{thm}\label{t7} Let $R$ be the solvable radical of a Leibniz
algebra $L$ over a field of characteristic zero. Then $d(R)
\subseteq R$ for any $d \in LDer_n(L).$
\end{thm}

\begin{proof} Let $d$ be a Leibniz-derivation of order $n.$
Due to Proposition \ref{p1}, $R=R_n,$ so it is enough to prove the
assertion of the Theorem for $R_n.$

Since $R_n$ is a $n$-solvable radical, there exists $s\in
\mathbb{N}$ such that $\mathcal{L}_n^{[s]}(R_n)=0.$ Then by Lemma
\ref{inclusion}, $\mathcal{L}_n^{[s]}(d(R_n)) \subseteq
R_n+d^{n^{s-1}}(\mathcal{L}_n^{[s]}(R_n))=R_n.$ Thus, we have
$\mathcal{L}_n^{[s]}(R_n + d(R_n)) \subseteq R_n.$

Further,
$$\mathcal{L}_n^{[2s-1]}(R_n + d(R_n)) \subseteq
\mathcal{L}_n^{[s]}(\mathcal{L}_n^{[s]}(R_n + d(R_n))) \subseteq
\mathcal{L}_n^{[s]}(R_n) =0.$$

The $n$-ideal property of $R_n +d(R_n)$ follows from the following equalities: $$[l_1,
\dots,l_i+d(l_i),\dots ,l_n]_r=[l_1, \dots,l_i,\dots ,l_n]_r+[l_1,
\dots,d(l_i),\dots ,l_n]_r=$$ $$[l_1, \dots,l_i,\dots ,l_n]_r+d([l_1,
\dots,l_i,\dots,l_n]_r)- \sum_{j=1, j\neq i}^n
[l_1,\dots ,d(l_j),\dots ,l_n]_r.$$

Hence $R_n +d(R_n)$ is an $n$-solvable ideal of the Leibniz algebra
$L.$ Since $R_n$ is the $n$-solvable radical of $L,$ it follows
that $R_n + d(R_n) \subseteq R_n$, therefore $d(R_n)\subseteq
R_n.$
\end{proof}

\begin{lem} \label{inclusion2} Let $I$ be an ideal of the
Leibniz algebra $L$ and $d \in LDer_n(L)$ a Leibniz-derivation
 for some $n \in \mathbb{N}.$
Then $$\mathcal{L}_n^{k}(d(I))
\subseteq I+d^{kn-k+1}(\mathcal{L}_n^{k}(I))$$ for all $k\in \mathbb{N}.$
\end{lem}
\begin{proof} For $k=1$ the assertion of the lemma is obvious. Let
$k=2$, then using the formula (3.1) we have
$$\mathcal{L}_n^{2}(d(I)) = [d(I), d(I), \dots, d(I)]_r \subseteq
d([I, I, \dots, I]_r) + $$
$$+\sum_{\begin{array}{c} i_1+i_2+\cdots+i_n=n\\\exists i_j=0\end{array}}\frac{n!}{i_1!i_2!\dots
i_n!}[d^{i_1}(I), \dots, d^{i_j-1}(I), I, d^{i_j-1}(I), \dots,d^{i_n}(I)]_r\subseteq$$
$$\subseteq I + d^{n}(\mathcal{L}_n^{2}(I)).$$

Assume that $\mathcal{L}_n^{k}(d(I)) \subseteq
I+d^{kn-k+1}(\mathcal{L}_n^{k}(I)).$ Applying the formula (3.1),
we prove the inclusion for $k+1:$
$$\mathcal{L}_n^{k+1}(d(I)) = [\mathcal{L}_n^{k}(d(I)), (d(I)), \dots, (d(I))]_r
\subseteq [I+d^{kn-k+1}(\mathcal{L}_n^{k}(I)), d(I), \dots, d(I)]_r
\subseteq$$
 $$  \subseteq [I + d^{kn-k+1 + n-1}([\mathcal{L}_n^{[k]}(I), I, \dots, I]_r)
 \subseteq I + d^{(k+1)n-k}(\mathcal{L}_n^{k+1}(I))).
$$

\end{proof}

Invariant property of nilradical of a Leibniz algebra under
a Leibniz-derivation is presented in the following theorem.

\begin{thm}\label{t8} Let $N$ be the nilradical of a Leibniz
algebra $L$ over a field of characteristic zero. Then $d(N)
\subseteq N$ for any $d \in LDer_n(L).$
\end{thm}

\begin{proof} The proof is similar to the proof of Theorem \ref{t7}.
\end{proof}

Next result establish  properties of weight spaces with respect
to a Leibniz-derivation of a Leibniz algebra.

\begin{lem}\label{l5} Let $L$ be a complex Leibniz algebra
with a given Leibniz-derivation $d$ of order $n$ and $L =
L_{\alpha}\oplus L_{\beta}\oplus \dots \oplus L_{\gamma}$ the
decomposition of $L$ into weight spaces with respect to $d$
(i.e. $L_{\alpha}= \{x \in L: (d-\alpha
1)^k x =0 \ \rm{for \   some} \ k\}$). Then
$$[L_{\alpha_1}, L_{\alpha_2}, \dots ,
L_{\alpha_n}]_r = \left\{\begin{array}{ll} 0 & if \  \alpha_1+\alpha_2 + \dots + \alpha_n  \ is \ not \ a \ root \ of \ d\\
L_{\alpha_1+\alpha_2 + \dots + \alpha_n} & if \ \alpha_1+\alpha_2 + \dots + \alpha_n \  is \ a \ root \ of \ d.\\ \end{array}\right.
$$
\end{lem}

\begin{proof}
First observe that
$$(d-(\alpha_1+\alpha_2+\cdots+\alpha_n)\cdot 1)[x_1,x_2,\dots,x_n]_r=
$$
$$\sum\limits_{i=1}^n[x_1, \dots, d(x_i), \dots, x_n ]_r -
\sum\limits_{i=1}^n[x_1, \dots, \alpha_i x_i, \dots, x_n ]_r =
\sum\limits_{i=1}^n[x_1, \dots, (d-\alpha_i \cdot 1) x_i, \dots,
x_n ]_r.$$

Similarly to Lemma \ref{derivacion}, by induction on $k$ we
get the following equality:
$$(d-(\alpha_1+\alpha_2+\cdots+\alpha_n)\cdot 1)^k[x_1,x_2,\dots,x_n]_r=
$$
$$=\sum_{i_1+i_2+\cdots+i_n=k}\frac{k!}{i_1!i_2!\cdots i_n!} [(d-\alpha_1
\cdot 1)^{i_1}x_1,(d-\alpha_2 \cdot 1)^{i_2}x_2,\dots,(d-\alpha_n
\cdot 1)^{i_n}x_n]_r\eqno (4.1)$$ for any $x_i\in L_{\alpha_i}$.

Consider $x_i \in L_{\alpha_i}, 1\leq i \leq n.$ Then there exist
natural numbers $k_i$ such that $(d-\alpha_i\cdot 1)^{k_i}(x) =0.$
 Taking $k=\sum\limits_{i=1}^n k_i$ in (4.1), we have
$$(d-(\alpha_1+\alpha_2+\cdots+\alpha_n)\cdot
1)^k[x_1,x_2,\dots,x_n]_r = 0$$ which completes the proof.
\end{proof}

Similarly as in \cite{Moens} we have the existence of an invertible
 Leibniz-derivation of nilpotent Leibniz algebra.

\begin{prop}\label{p5}
Every nilpotent Leibniz algebra with nilindex equal to $s$ has an
invertible Leibniz-derivation of order $[\frac s 2]+1.$
\end{prop}

\begin{proof} Let $L$ be a Leibniz algebra with nilindex equal to $s$ and set $q=[\frac s 2]+1.$
Consider the vector subspace $W$ of $L$ complementary to $L^{q},$
i.e. $L = W + L^{q}$. Define the map $P$ by the following way:
$$P(x) = \left\{\begin{array}{ll} x & if \ x\in W,\\  qx & if \ x\in L^{q}.\\ \end{array}\right.$$
It is easy to check that $P$ is a Leibniz-derivation for $L$ of order $q.$
\end{proof}

Below we present one of the main theorems of the paper.

\begin{thm}\label{t6}
Let $L$ be a complex Leibniz algebra which admits an invertible
Leibniz-derivation. Then $L$ is nilpotent.
\end{thm}

\begin{proof}
Let $d$ be an invertible Leibniz-derivation of order $n$ of the Leibniz
algebra $L$ and $$L=L_{\rho_1}\oplus L_{\rho_2}\oplus\dots\oplus
L_{\rho_s}$$ be the decomposition of $L$ into characteristic spaces
with respect to $d.$

Let $\alpha,\beta \in spec(d).$ Then by Lemma \ref{l5} we have
$$[L_{\alpha},\underbrace{ L_{\beta}, L_{\beta}, \dots, L_{\beta}
]_r}_{n-1-\textit{times}}= [\dots[[L_{\alpha},\underbrace{
L_{\beta}], L_{\beta},] \dots, L_{\beta} ]}_{n-1-\textit{times}}
\subseteq L_{\alpha+(n-1)\beta}.$$ Considering $k$-times of the
$n$-ary multiplication we have
$$[\dots[[L_{\alpha},\underbrace{ \underbrace{L_{\beta}, L_{\beta}, \dots, L_{\beta}
]_r}_{n-1 - \textit{times}}, \underbrace{L_{\beta}, L_{\beta},
\dots, L_{\beta}]_r}_{n-1 - \textit{times}}, \dots,
\underbrace{L_{\beta}, L_{\beta}, \dots, L_{\beta}]_r}_{n-1 -
\textit{times}}}_{k-\textit{times}}= $$
$$= [\dots[[L_{\alpha},\underbrace{
L_{\beta}], L_{\beta},] \dots, L_{\beta}
]}_{k(n-1)-\textit{times}} \subseteq L_{\alpha+k(n-1)\beta}.$$

Since for sufficiently large $k\in \mathbb{N}$ we obtain
$\alpha+k(n-1)\beta\not \in spec(d),$ by Lemma \ref{l5} we obtain
$[\dots[[L_{\alpha},\underbrace{ L_{\beta}], L_{\beta},] \dots,
L_{\beta} ]}_{k(n-1)-\textit{times}}=0.$

Thus, any operator of right multiplication $R_x: L \rightarrow L$,
where $x\in L_{\beta}$, is nilpotent and, due to the fact that
$\alpha, \beta$ were taken arbitrary, it follows that every
operator from $\bigcup_{i=1}^kR(L_{\rho_i})$ is nilpotent.

Now from identity $(2.1)$ and Lemma \ref{l5} it follows that
$\bigcup_{i=1}^kR(L_{\rho_i})$ is a weekly closed set of an
associative algebra $End(L).$ Hence, by Theorem \ref{closed} it
follows that every operator from $R(L)$ is nilpotent.

Hence, $R_x$ is nilpotent for any $x \in L.$ Now by Engel's
Theorem (Theorem \ref{t1}) we conclude that $L$ is
nilpotent.\end{proof}

Finally from the Theorem \ref{t6} and Proposition \ref{p5} we get
the analogue of Theorem \ref{t5} for Leibniz algebras.

\begin{thm} A Leibniz algebra over a field of characteristic zero is nilpotent
if and only if it has an invertible Leibniz-derivation.
\end{thm}

\end{document}